\documentclass[twoside]{article}

\usepackage{amsfonts}
\usepackage{amsthm}
\usepackage{amsmath}
\usepackage{amssymb}
\usepackage[utf8]{inputenc}
\usepackage{graphicx}
\usepackage{fancyhdr}
\usepackage{hyperref}
\usepackage{geometry}
\usepackage{enumerate}
\usepackage{xcolor}
\usepackage{authblk}
\usepackage{subcaption}

\newcommand{\R}{{\mathbb R}}
\newcommand{\D}{{\mathbb D}}
\newcommand{\N}{{\mathbb N}}
\renewcommand{\k}{{\bf k}}
\newcommand{\e}{{\bf e}_1}
\newcommand{\ed}{{\bf e}_2}
\newcommand{\pd}[2]{\frac{\partial #1}{\partial #2}}
\newcommand{\wt}{\widetilde}
\newcommand{\spk}{\operatorname{Sup}}
\newcommand{\diam}{\operatorname{diam}}

\newcommand{\calI}{\mathcal{I}}
\newcommand{\calP}{\mathcal{P}}
\newcommand{\calV}{\mathcal{V}}

\newcommand{\frakD}{\mathfrak{D}}
\newcommand{\frakI}{\mathfrak{I}}
\newcommand{\frakP}{\mathfrak{P}}

\theoremstyle{plain}
\newtheorem{theorem}{Theorem}[section]
\newtheorem{corollary}[theorem]{Corollary}
\newtheorem{lemma}[theorem]{Lemma}

\theoremstyle{definition}
\newtheorem{definition}[theorem]{Definition}
\newtheorem{example}[theorem]{Example}

\theoremstyle{remark}
\newtheorem{remark}[theorem]{Remark}

\definecolor{org1}{RGB}{67, 132, 252}
\hypersetup{
  colorlinks=true,
  allcolors=blue
}

\begin{document}
\title{Hyperbolic Functions of Bounded Variation and Riemann-Stieltjes Integral involving Strong Partitions of Hyperbolic Intervals}
\author{Gamaliel Yafte Tellez-S\'anchez$^{(\dagger)}$, Juan Bory-Reyes$^{(\star)}$}
\date{\small $^{(\dagger)}$ Escuela Superior de F\'isica y Matem\'aticas. Instituto Polit\'ecnico Nacional. Edif. 9, 1er piso, U.P. Adolfo L\'opez Mateos. 07338, Mexico City, M\'exico.\\E-mail: gtellez.wolf@gmail.com\\
		$^{(\star)}$ Escuela Superior de Ingenier\'ia Mec\'anica y El\'ectrica. Instituto Polit\'ecnico Nacional. Edif. 5, 3er piso, U.P. Adolfo L\'opez Mateos. 07338, Mexico City, MEXICO\\ Email: juanboryreyes@yahoo.com}

\maketitle

\begin{abstract}
In this paper, we define two types of partitions of an hyperbolic interval: weak and strong. Strong partitions enables us to define, in a natural way, a notion of hyperbolic valued functions of bounded variation and hyperbolic analogue of Riemann-Stieltjes integral. We prove a deep relation between both concepts like it occurs in the context of real analysis.

\bigskip
\noindent
{\bf Math Subject Classification (2020):}: 30G35, 28B15, 26B15.

\end{abstract}

\section{Introduction}

Literature on hyperbolic numbers like, for instance, \cite{vd1935, tb2021} highlights that these numbers themselves are not so complicated and have hybrid behavior between real numbers and complex numbers. The hyperbolic numbers born as a real sub-algebra of Tessarine numbers introduced by J. Cockle in \cite{c1848} who focused the study in an hypercomplex analysis context with a Cauchy-Riemann type system. However, there exists an isomorphism between the hyperbolic numbers plane and the Cartesian product of the real numbers set with itself that changes the hypercomplex direction to more similar real analysis framework.

With this aim, a partial order was introduced in hyperbolic numbers, see the classical reference \cite{s1995}, which provides conditions of holomorphicity and continuity for functions of hyperbolic variable early presented in \cite{vd1935}, see also \cite{kgs2020}. Partial order is used to prove that natural domains are rectangles in the hyperbolic plane.

In recent years, we have seen the great success of hyperbolic intervals partitioning. For example, in \cite{bbls2016, tb2017, t2018} the concept was applied to certain classification of Cantor type sets in the hyperbolic plane. An special example was considered in \cite{ tb2021, als2017}, where the focus is to consider sets of probabilities defined like a division of the unit hyperbolic interval $[0, \wt{1}]_{\k}$.

To understand how an hyperbolic interval can be divided in Section \ref{scHP}, two types of partitions are defined, where the prominent type: strong partition, allows us to concern in the later sections with the notion of natural hyperbolic functions.

Sections \ref{scHFBV} and \ref{scHRSI} are devoted to important applications of strong partitions. First, we introduce a notion of hyperbolic functions of bounded variation and indicate how the set of discontinuities may be defined like in real numbers occurs. Second, we proceed with the definition of Riemann-Stieltjes integral of a hyperbolic valued functions and its relation with the derived Riemann integral presented in \cite{vd1935, ks2017}.

As well as in real numbers happen, see \cite{a1974}, the hyperbolic bounded variation condition introduced here is sufficient for the existence of the hyperbolic Riemann-Stieltjes integral. This is the final conclusion of Section \ref{scHRSI}.

\section{Hyperbolic Numbers}
Hyperbolic numbers are a generalization of complex numbers, which are classical extension of the real numbers by the inclusion of an imaginary unit, whereas the hyperbolic numbers also do it by a new square root $\k \not\in \R$ such that $\k^2 = 1$ . These are introduced by J. Cockle in \cite{c1848} like a sub-algebra from the nowadays well know bicomplex numbers, see \cite{vd1935, s1995}. Whereas each nonzero complex number has a multiplicative inverse, this is no longer true for all nonzero hyperbolic numbers.

The real ring of hyperbolic numbers is the commutative ring denoted usually by
\[\D := \R[\k] = \{ a + b\k \ |\ a, b \in \R \}.\]
Because there exists a bijection between Euclidean plane and hyperbolic numbers, set $\D$ is also known as hyperbolic numbers plane. 

There are two very special zero divisors and mutually complementary elements in $\D$ which are
\[\e = \frac{1+\k}{2}, \quad \ed = \frac{1 - \k}{2}.\]

Given $\alpha \in \D$ we have
\[\alpha = t + s\k \Rightarrow \alpha = (t + s)\e + (t - s)\ed. \]
\[\alpha = a_1\e + a_2\ed \Rightarrow \frac{1}{2}(a_1 + a_2) + \frac{1}{2}(a_1 - a_2)\k. \]
Therefore, $\D = \R\e + \R\ed$ and a ring isomorphism maps $\D$ into the direct product $\R \oplus \R$. In this way we obtain what   will be referred to as idempotent representation.

Real line is endowed into the hyperbolic plane by the function $x \mapsto \wt{x} = x\e + x\ed$. On the other side the idempotent projections of a subset $A \subset \D$ are the real sets
\[A_{\e} := \{ a \in \R \ |\ \exists b\in \R,\ a\e + b\ed \in A \}, \]
\[A_{\ed} := \{ b \in \R \ |\ \exists a \in \R,\ a\e + b\ed \in A \}. \]

\subsection{Partial order}
Hyperbolic numbers are a partially ordered set with a binary relation given by 
\[\alpha \preceq \beta \quad \Leftrightarrow a_1 \leq b_1\ \land\ a_2 \leq b_2,\]
for $\alpha, \beta \in \D$, where $\alpha = a_1\e + a_2\ed$ and $\beta = b_1\e + b_2\ed$. 

The strict order is defined in a similar way, indicating the strict order in the real line
\[\alpha \prec \beta \quad \Leftrightarrow \quad a_1 < b_1 \ \land \ a_2 < b_2. \]

The strict order is not partial order with the constraint that $\alpha \not= \beta$, since it implies that one of the following two cases could occur: nor $a_1 = b_1$ and $a_2 \not= b_2$,  or $a_1 \not= b_1$ and $a_2 = b_2$. These cases have been deleted in the strict order.

Partial order give the possibility to define in a natural way different types of intervals. If $\alpha$ and $\beta$ are related number such that $\alpha \prec \beta$, then an hyperbolic closed interval is defined by
\begin{displaymath}
[\alpha, \beta]_{\k} := \{ \xi \in \D \ |\ \alpha \preceq \xi \preceq \beta \}.
\end{displaymath}
Likewise, an open hyperbolic interval setting up 
\begin{displaymath}
(\alpha, \beta)_{\k} := \{ \xi \in \D \ |\ \alpha \prec \xi \prec \beta \}.
\end{displaymath}
We will consider others type of intervals, for instance, $[\alpha, \beta)_{\k}$ and $(\alpha, \beta]_{\k}$ to be defined in much the same way.

By the length of an hyperbolic interval $\frakI = [\alpha, \beta]_{\k}$ (equally valid for all types of intervals) we mean
\begin{displaymath}
\lambda_{\k}(\frakI) = \beta - \alpha
\end{displaymath}

\subsection{Natural hyperbolic functions}
Let $\Omega \subset \D$ be a domain. A function $F: \Omega \rightarrow \D$ may also be viewed as functions of two real variables on the Euclidean plane. Into hyperbolic numbers plane there exists a well defined Cauchy-Riemann equations theory (see \cite{vd1935, mr1998, ks2005}). If $F = u + v\k$ is differentiable in $\xi_0 = t_0 + s_0\k$, it fulfill
\[ \frac{\partial u}{\partial t}(\xi_0) = \frac{\partial v}{\partial s}(\xi_0), \quad \frac{\partial u}{\partial s}(\xi_0) = \frac{\partial v}{\partial t}(\xi_0).\]
Moreover, if $F = F_1\e + F_2\ed$ satisfies the Cauchy-Riemann equations, the idempotent components of $F$ are functions of one real variable, see   \cite{vd1935, tb2019} for more details. This implies that the derivative of $F$, denoted by $F'$, is then computed by the partial derivatives on every component, which are taken as the total derivatives at the point $\xi_0 = x_0\e + y_0\ed$.
\[F'(\xi_0) = \pd{F_1}{x}(\xi_0)\e + \pd{F_2}{y}(\xi_0)\ed = \frac{dF_1}{dx}(x_0)\e + \frac{dF_2}{dy}(y_0)\ed.\]
It is easy to check that every domain $\Omega$ can be extended to the minimum open interval that contains it, for a fuller treatment \cite{vd1935}.
\[\overline{\Omega} := \Omega_{\e}\e + \Omega_{\ed}\ed.\]
Hence, $F$ will be defined on $\overline{\Omega}$ and if it is regarded in the idempotent representation $F = F_1\e + F_2\ed$, then $F_1$ and $F_2$ are real valued functions over $\Omega_{\e}$ and $\Omega_{\ed}$ respectively.

By this reason we realize that functions $F: \frakI \rightarrow \D$ are the natural subject of study in the hyperbolic plane, where $F = F_1\e + F_2\ed$, $F_1$, $F_2$ are real valued functions of one variable and $\frakI$ can be either an open or a close hyperbolic interval. Therefore, under previous features, we shall call $F$ a natural hyperbolic function.

\subsection{Continuous hyperbolic functions}\label{sc1}
With the partial order in the hyperbolic numbers, several concepts can be extended to objects defined in the hyperbolic plane. For example in \cite{kgs2020, ks2016} the concept of hyperbolic metric spaces was studied. 

The duple $(X, D)$ is an hyperbolic metric space when $X$ is a no-empty set and $D: X \times X \rightarrow D_{0}^{+}$ is a function with the next requirements for all $x, y, z \in X$,
\begin{enumerate}[1)- \quad]
\item $D(x, y) = 0 \Leftrightarrow x = y$.
\item $D(x, y) = D(y, x)$.
\item $D(x, y) \preceq D(x, z) + D(z, y)$.
\end{enumerate}

Hyperbolic numbers form an hyperbolic metric space with the usual hyperbolic metric $|\cdot |_{\k}: \D \times \D \rightarrow \D_{0}^{+}$  such that for every $\xi, \gamma \in \D$ and $\xi = x_1\e + x_2\ed$, $\gamma = y_1\e + y_2\ed$, we have
\[|\xi - \nu|_{\k} = |x_1 - y_1|\e + |x_2 - y_2|\ed.\]
This fact has previously been introduced in \cite{kgs2020, ks2016, lssv2015}.

Continuity of functions between two hyperbolic metric spaces is a concept already treated in the literature, see \cite{kgs2020}. A function $F:(X, D_X) \rightarrow (Y, D_Y)$ between hyperbolic metric spaces is say to be continuous, if for every hyperbolic positive number $\epsilon \in \D_{+}$ there exists a $\delta \in \D_{+}$ such that for every $\xi, \nu \in X$, with $D_X(\xi, \nu) \prec \delta$ we have that $D_Y(F(\xi), F(\nu)) \prec \epsilon$.

A natural hyperbolic function $F:[\alpha, \beta]_{\k}\rightarrow \D$ in idempotent representation, with usual hyperbolic metric on $\D$, is continuous if and only if every component $F_j: [a_j, b_j] \rightarrow \R, j \in \{1, 2\}$ is real continuous.

\section{Partitions Involving Hyperbolic Intervals}\label{scHP}
Let us start with a brief discussion of the possibility to recover all information about an hyperbolic interval after a division in sub-intervals. For instance, the length of the original hyperbolic interval. This can be found in \cite{bbls2016}.
\begin{example}\label{ex1}
Considering the interval $\frakI = [0, \wt{1}]_{\k}$ and dividing it by nine sub-intervals
\begin{displaymath}
\arraycolsep=1.4pt\def\arraystretch{2.2}
\begin{array}{ccc}
\frakI_1 = \left[0, \wt{\frac{1}{3}}\right]_{\k}, & \frakI_2 = \left[\wt{\frac{1}{3}}, \wt{\frac{2}{3}} \right]_{\k}, & \frakI_3 = \left[\wt{\frac{2}{3}}, \wt{1} \right]_{\k}, \\
\frakI_4 = \left[\frac{1}{3}\e, \frac{2}{3}\e + \frac{1}{3}\ed \right]_{\k}, & \frakI_5 = \left[\frac{2}{3}\e, 1\e + \frac{1}{3}\ed \right]_{\k}, & \frakI_6 = \left[\frac{2}{3}\e + \frac{1}{3}\ed, 1\e + \frac{2}{3}\ed \right]_{\k}, \\
\frakI_7 = \left[\frac{1}{3}\ed, \frac{1}{3}\e + \frac{2}{3}\ed \right]_{\k}, & \frakI_8 = \left[\frac{2}{3}\ed, \frac{1}{3}\e + 1\ed \right]_{\k}, & \frakI_9 = \left[\frac{1}{3}\e + \frac{2}{3}\ed, \frac{2}{3}\e + 1\ed \right]_{\k}.
\end{array}
\end{displaymath}
Therefore
\begin{displaymath}
\lambda_{\k}(\frakI) = \wt{1} \neq \wt{3} = \sum_{n = 1}^{9} \lambda_{\k}(\frakI_n)
\end{displaymath}
\end{example}

The Example \ref{ex1} shows a regular partition of an square in the Euclidean plane. It fulfill, in Lebesgue sense, that the sum of the areas of every sub-squares is equal to the total area of the big square.

By this reason, an interval $[\alpha, \beta]_{\k}$ may be considered as a square in the Euclidean plane. A partition in sub-rectangles $S_1, ..., S_k$ such that
\[\mu_{\R}([\alpha, \beta]_{\k}) = \sum_{j = 1}^{k} \mu_{\R}(S_j), \]
with $\mu_{\R}$ denotes the Lebesgue measure, will be called a regular partition.
\begin{remark}
Previous definition is not restricted to partitions of rectangles generated by real intervals. For every $j \in \{1,..., k\}$, $S_j$ could be a measurable set and if $j \neq t$, then $S_j \cap S_t$ has measure zero.
\end{remark}
In order to guarantee a positive answer we make the following natural definition of hyperbolic interval partition.
\begin{definition}
A collection $\calI$ of sub-intervals from $[\alpha, \beta]_{\k}$ is a weak partition when 
\[ \lambda_{\k}([\alpha, \beta]_{\k}) = \sum_{I \in \calI} \lambda_{\k} (I). \]
\end{definition} 

This type of partitions has the disadvantage of being able to be constituted by disjoint sub-intervals as Fig. \ref{fg:wp} shows.
\begin{figure}[ht]
\begin{subfigure}{.5\textwidth}
\centering
\includegraphics[scale=1]{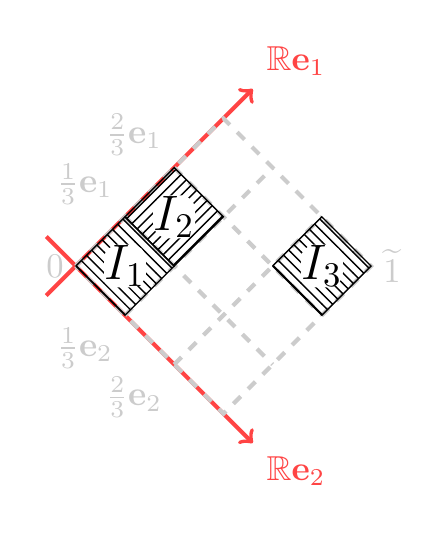}
\caption{ }
\label{fg:wp1}
\end{subfigure}
\begin{subfigure}{.5\textwidth}
\centering
\includegraphics[scale=1]{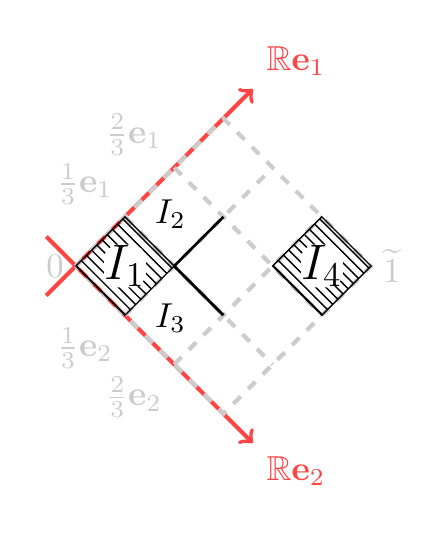}
\caption{ }
\label{fg:wp2}
\end{subfigure}
\caption{Example of weak partition}
\label{fg:wp}
\end{figure}

Figure \ref{fg:wp1} has three sub-intervals from $[0, \wt{1}]_{\k}$. All of them have length equal to $\wt{\frac{1}{3}}$. Therefore
\[\lambda_{\k}(I_1) + \lambda_{\k}(I_2) + \lambda_{\k}(I_3) = \wt{1}.\]

In the other hand Fig. \ref{fg:wp2} has four intervals where $I_1$ and $I_4$ have length equal to $\wt{\frac{1}{3}}$, but for the two remaining intervals we have
\[\lambda_{\k}(I_2) = \frac{1}{3}\e \text{ and } \lambda_{\k}(I_3) = \frac{1}{3}\ed. \]
So, the sum of lengths of the four intervals being equal to $\wt{1}$.
\paragraph{}
In \cite{bbls2016}, to avoid the disjoint intervals issue, a condition under which a collection of points into an hyperbolic interval provides a collection of sub-intervals whose lengths add up to the length of the biggest interval and do not have empty intersection was established there. 

\begin{definition} \label{strongpartition}
Let $\frakP = \{\rho_0, ..., \rho_n\}$ be a finite collection of points in the interval $[\alpha, \beta]_{\k}$ such that $\rho_s \neq \rho_t$ when $s \neq t$. We say that $\frakP$ is a strong partition, if both conditions are fulfill
\begin{enumerate}[1)- \quad]
\item $\frakP$ is a chain on $\D$.
\item\label{strongpartition2} $\rho_0 = \alpha$, $\rho_n = \beta$ and
\[ \rho_0 \preceq \rho_1 \preceq ... \preceq \rho_n .\]
\end{enumerate}
\end{definition}
There are two differences between Def. \ref{strongpartition} and that given in \cite{bbls2016}. The first is that equality in (\ref{strongpartition}-\ref{strongpartition2}) is avowed, meanwhile in \cite{bbls2016} an strict relation is required. As a consequence a third condition relative to the absence of zero divisor in the lengths among sub-intervals of the kind $[\rho_{j-1}, \rho_{j}]_{\k}$ is established. Inclusion of equality into the second requirement do not alter the proof of the next theorem.
\begin{theorem}
If $\frakP$ is a strong partition of $[\alpha, \beta]_{\k}$, then
\[ \sum_{j = 1}^{n} \lambda_{\k}([\rho_{j-1}, \rho_{j}]_{\k}) = \lambda_{\k} ([\alpha, \beta]_{\k}). \]
\end{theorem}

As Fig. \ref{fg:sp1} shows, Definition \ref{strongpartition} enables degenerate sub-intervals to be built, where $I_2$ and $I_3$ are of this kind of interval. While Fig. \ref{fg:sp2} is the extension of the uniform real partition with norm equal to $\displaystyle\frac{1}{3}$.

\begin{figure}[ht]
\begin{subfigure}{.5\textwidth}
\centering
\includegraphics[scale=1]{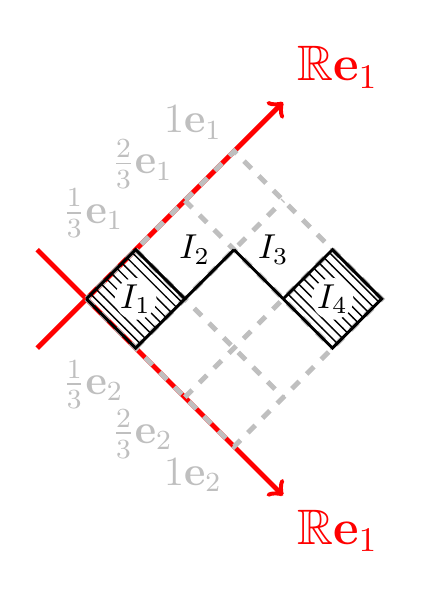}
\caption{ }
\label{fg:sp1}
\end{subfigure}
\begin{subfigure}{.5\textwidth}
\centering
\includegraphics[scale=1]{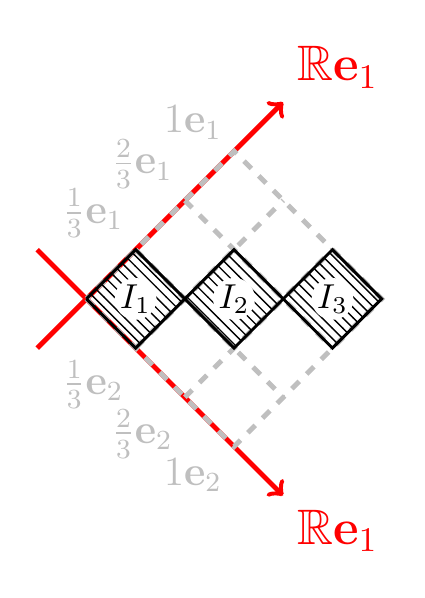}
\caption{ }
\label{fg:sp2}
\end{subfigure}
\caption{Example of strong partition}
\label{fg:sp}
\end{figure}
Strong partitions can generate partitions on real interval, only from the points that define it. So, if $\frakP = \{\rho_1, \rho_2, ..., \rho_n \}$ is an strong partition, let the projection sets
\begin{equation}\label{eq2}
\begin{split}
\frakP_{\e} & = \{ p_{1, 1}, p_{2, 1}, ..., p_{n, 1} \} \\
\frakP_{\ed} & = \{p_{1, 2}, p_{2, 2}, ..., p_{n, 2} \},
\end{split}
\end{equation}
where $\rho_{j} = p_{j, 1}\e + p_{j, 2}\ed$ for every $j \in \{1, 2, ..., n\}$.

\bigskip
It is possible to build real partitions with regular and weak partitions by projections of the endpoints of every intervals. Our interest is in strong partitions because they are in spirit similar to that of the real intervals context.

\subsection{Real partitions define a strong partition}\label{st2.1}

Two real intervals $[a_1, b_1]$ and $[a_2, b_2]$ define the hyperbolic interval $[\alpha, \beta]_{\k}$ with $\alpha = a_1\e + a_2\ed$ and $\beta = b_1\e + b_2\ed$. So, a natural question arise, How can we create a strong partition from two real partitions $P = \{p_0, p_1, ..., p_s\} \subset [a_1, b_1]$ and $Q = \{q_0, q_1, ..., q_t\} \subset [a_2, b_2]$?. 

By definition of strong partition, we need that initial and final points match with $\alpha$ and $\beta$. Therefore we have 
$\alpha = \rho_{0, 0} = p_0\e + q_0\ed$ and $\beta = \rho_{s, t} = p_s\e + q_t\ed$.

The general process to get an hyperbolic point is taking points $p_{s_j} \in P$ and $q_{t_j} \in Q$ with $p_{s_j - 1} \leq p_{s_j}$ and $q_{t_j - 1} \leq q_{t_j}$, but if $p_{s_j} = p_{s_j - 1}$, then $q_{t_j} \in Q \setminus \{q_{t_j - 1}\}$, in a similar way if $q_{t_j} = q_{t_j -1}$, then $p_{s_j} \in P \setminus \{p_{s_j - 1}\}$. We define $\rho_{s_j, t_j} = p_{s_j}\e + q_{t_j}\ed$.

Previous step only can be repeated in a maximum of $s + t$ times. And it finishes when $\rho_{s_j, t_j} = \rho_{s, t}$.

This procedure generates a strong partition $\frakP = \{\rho_{0, 0}, \rho_{s_1, t_1}, ..., \rho_{s, t} \}$. Figure \ref{realp} shows some examples.

\begin{figure}[ht]
\begin{subfigure}{.5\textwidth}
\centering
\includegraphics[scale=0.7]{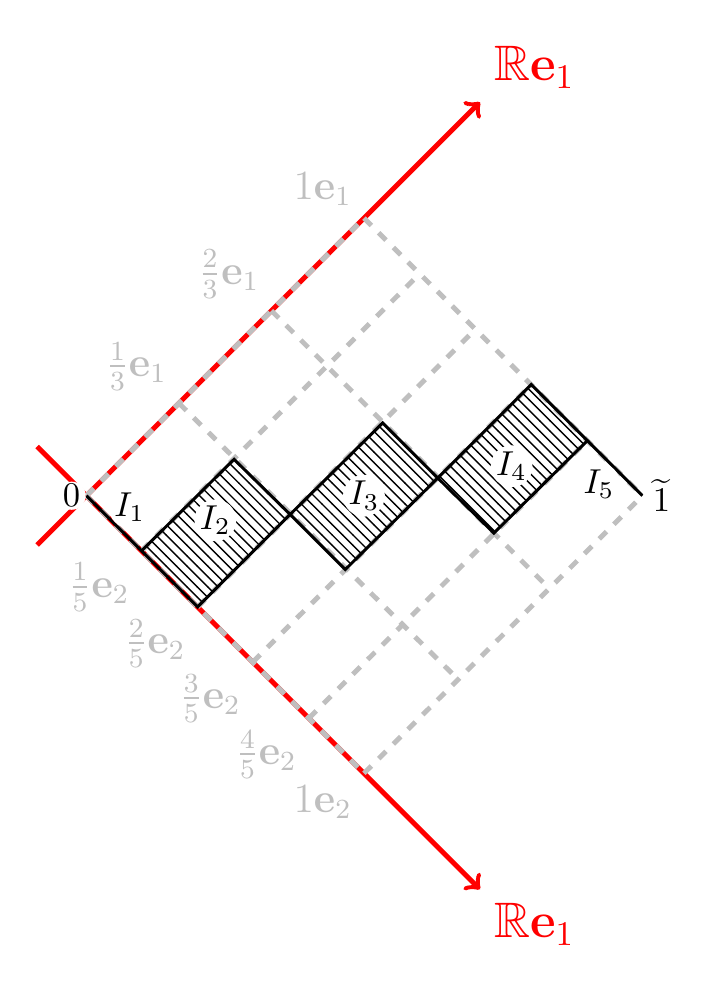}
\caption{ }
\label{realp1}
\end{subfigure}
\begin{subfigure}{.5\textwidth}
\centering
\includegraphics[scale=0.7]{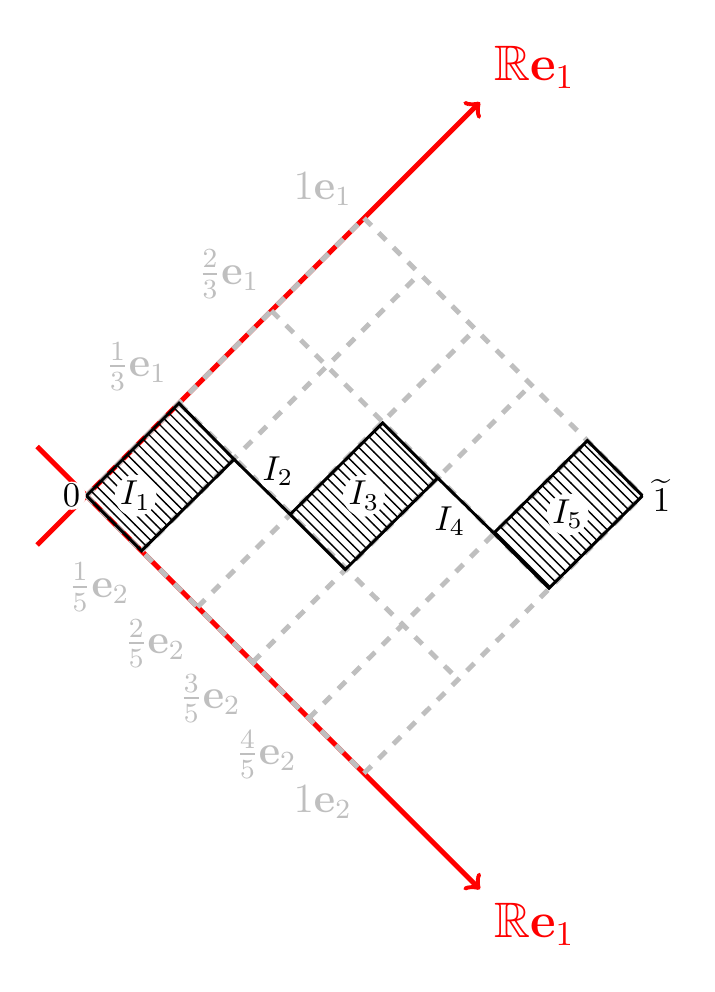}
\caption{ }
\label{realp2}
\end{subfigure}
\caption{Strong partition generated by $\displaystyle P_1 = \left\{0, \frac{1}{3}, \frac{2}{3}, 1\right\}$, $\displaystyle P_2 = \left\{0, \frac{1}{5}, \frac{2}{5}, \frac{3}{5}, \frac{4}{5}, 1\right\}$. Both are partitions on the real interval $[0, 1]$.}
\label{realp}
\end{figure}

\section{Hyperbolic Functions of Bounded Variation}\label{scHFBV}
In this section, the concept of hyperbolic valued functions of bounded variation is introduced. To do this, we give a brief exposition of the notion of supremum of a set in the hyperbolic plane.

\begin{definition}\label{d3}
Let $A$ be a no empty subset of $\D$. The supremum of $A$ is defined to be the number
\[\spk(A) := \sup(A_{\e})\e + \sup(A_{\ed})\ed. \]
\end{definition}

Definition \ref{d3} agrees with the idea that the supremum of $A$ is the least element in $\D$ that is greater than or equal to all elements of $A$. Even when there are elements into the set that are no related, all of them are related with the supremum by the partial order. For a fuller treatment we refer the reader to \cite{tb2017, t2018, lps2014, ssk2020}.

Let $[\alpha, \beta]_{\k} \subset \D$ and $F: [\alpha,  \beta]_{\k} \rightarrow \D$ an hyperbolic valued function with idempotent representation $F = F_{1}\e + F_{2}\ed$. 

If $\frakP = \{\rho_0, \rho_1, ..., \rho_{n_\frakP}\}$ is a strong partition of $[\alpha, \beta]_{\k}$, then we can consider the quantity
\begin{equation}\label{eq1}
\begin{split}
\sum_{j = 0}^{n_\frakP-1} |F(\rho_{j+1}) & - F(\rho_{j})|_{\k} = \\
& \sum_{j=0}^{n_\frakP-1} \left( |F_1(\rho_{j + 1}) - F_1(\rho_{j})|\e + |F_2(\rho_{j + 1}) - F_2(\rho_{j})|\ed \right) =  \\
& \left( \sum_{j=0}^{n_\frakP-1} |F_1(\rho_{j + 1}) - F_1(\rho_{j})| \right)\e + \left( \sum_{j=0}^{n_\frakP-1} |F_2(\rho_{j + 1}) - F_2(\rho_{j})| \right)\ed  = \\
& \left( \sum_{j=0}^{n_\frakP-1} \Delta_{\frakP, j}F_1 \right) \e + \left( \sum_{j=0}^{n_\frakP-1} \Delta_{\frakP, j}F_2 \right)\ed.
\end{split}
\end{equation}

Let $\calP([\alpha, \beta]_{\k})$ denote the family of all strong partitions for $[\alpha, \beta]_{\k}$ and we define the set
\[ \sum_{\calP([\alpha, \beta]_{\k})} (F) := \left\{ \left( \sum_{j=0}^{n_\frakP-1} \Delta_{\frakP, j}F_1 \right) \e + \left( \sum_{j=0}^{n_\frakP-1} \Delta_{\frakP, j}F_2 \right)\ed \  \big| \ \frakP \in \calP([\alpha, \beta]_{\k})  \right\}. \]

\begin{definition}\label{d4}
We say that an hyperbolic valued function $F: [\alpha, \beta]_{\k} \rightarrow \D$ is of bounded variation, when 
\[\spk\left( \sum_{\calP([\alpha, \beta]_{\k})} (F) \right) \prec \infty.\]
\end{definition}

Previous statement is equivalent to say that
\[\left( \sum_{\calP([\alpha, \beta]_{\k})} (F) \right)_{\e} \text{ and } \quad  \left( \sum_{\calP([\alpha, \beta]_{\k})} (F)\right)_{\ed}. \]
are bounded sets in the real line.

\begin{definition}\label{d5}
The total variation of a function $F: [\alpha, \beta]_{\k} \rightarrow \D$ of bounded variation, is the quantity
\[\calV_{[\alpha, \beta]_{\k}}(F) := \spk\left( \sum_{\calP([\alpha, \beta]_{\k})} (F) \right).\]
\end{definition}

According to Def. \ref{d3}, we have 
\[
\begin{split}
\calV_{[\alpha, \beta]_{\k}}(F) & = \spk \left( \sum_{\calP([\alpha, \beta]_{\k})} (F) \right) \\
& = \sup_{\frakP \in \calP([\alpha, \beta]_{\k})}  \left( \sum_{j=0}^{n_\frakP-1} \Delta_{\frakP, j}F_1 \right)\e + \sup_{\frakP \in \calP([\alpha, \beta]_{\k})}  \left( \sum_{j=0}^{n_\frakP-1} \Delta_{\frakP, j}F_2 \right)\ed
\end{split}
\]

Due the bijection from $\D$ to $\R^2$, a function $F = F_1\e + F_2\ed$ may be viewed as a function of two real variables in the Euclidean plane. Therefore, Def. \ref{d4} implies that $F = (F_1, F_2)$ is a function of bounded Vitali variation (see \cite{v1908, f1910, ca1933, ad2015}).

Because every component of a natural hyperbolic function relies from the respective component in the point, the sum in Eq. \ref{eq1} is simplified. So, if $F: [\alpha, \beta]_{\k} \rightarrow \D$ is a natural hyperbolic function, then 
\begin{equation}\label{eq3}
\begin{split}
\sum_{j = 0}^{n_\frakP-1} |F(\rho_{j+1}) & - F(\rho_{j})|_{\k}  = \\
& \left( \sum_{j=0}^{n_\frakP-1} |F_1(p_{j + 1, 1}) - F_1(p_{j, 1})| \right)\e + \left( \sum_{j=0}^{n_\frakP-1} |F_2(p_{j + 1, 2}) - F_2(p_{j, 2})| \right)\ed
\end{split}
\end{equation}

\begin{remark}\label{rk1}
Equation \ref{eq2} implies that the sums in Eq. \ref{eq1}  are taken over the projections $\frakP_{\e}$, $\frakP_{\ed}$, which are partitions of the real intervals $\left( [\alpha, \beta]_{\k} \right)_{\e} = [a_1, b_1]$ and $\left( [\alpha, \beta]_{\k} \right)_{\ed} = [a_2, b_2]$ respectively. 
\end{remark}

Let us denote by $\calP([a_{j}, b_{j}])$ the collection of all partitions of the real interval $[a_{j}, b_{j}]$ for every $j \in \{1, 2\}$ and introduce the sets
\[ \begin{split}
\sum_{\calP([a_{j}, b_{j}])}(F_{j}) = \left\{ \sum_{j=0}^{n_P-1} |F_j(p_{j + 1}) - F_j(p_{j})| \ |\ P \in \calP([a_{j}, b_{j}]) \right\}.
\end{split} \]

\begin{theorem}\label{pt1}
If $F:[\alpha, \beta]_{\k} \rightarrow \D$ is a natural hyperbolic function, then 
\[ \sum_{\calP([\alpha, \beta]_{\k})} (F) = \sum_{\calP([a_{1}, b_{1}])}(F_{1})\e + \sum_{\calP([a_{2}, b_{2}])}(F_{2})\ed. \]
\end{theorem}
\begin{proof}
Taking an element in the set $\displaystyle  \sum_{\calP([\alpha, \beta]_{\k})} (F)$, by Remark \ref{rk1}, the projections $\frakP_{\e}$ and $\frakP_{\ed}$ are partitions over $[a_1, b_1]$ and $[a_2, b_2]$ respectively and it has the form in the Eq. \ref{eq3}.

Reciprocally, two partitions $P \in \calP([a_{1}, b_{1}])$ and $Q \in \calP([a_{2}, b_{2}])$ define a strong partition $\frakP$ (see Section \ref{st2.1}). Partition $\frakP$ fulfill with $\frakP_{\e} = P$ and $\frakP_{\ed} = Q$, therefore even if the process in Sec. \ref{st2.1} generates $n_{\frakP} =n_{P} + n_{Q}$ points, where $n_{P}$ and $n_{Q}$ denote the cardinality of $P$ and $Q$, no additional elements in the sum are added, since in degenerated intervals $|F_1(p_{j + 1}) - F_1(p_{j})| = 0$ or $|F_2(p_{j + 1}) - F_2(p_{j})| = 0$, implying that
\[ \begin{split}
\left( \sum_{j=0}^{n_\frakP-1} \Delta_{\frakP, j}F_1 \right) \e & + \left( \sum_{j=0}^{n_\frakP-1} \Delta_{\frakP, j}F_2 \right)\ed =  \\
& \left(\sum_{j=0}^{n_P-1} |F_1(p_{j + 1}) - F_1(p_{j})| \right)\e + \left( \sum_{j=0}^{n_Q-1} |F_2(q_{j + 1}) - F_2(q_{j})| \right)\ed. 
\end{split} \]
\end{proof}
Combining Def. \ref{d4} with Thm. \ref{pt1}, hyperbolic valued functions of bounded variation are constructed.
\begin{corollary}\label{crBVC}
Let $F:[\alpha, \beta]_{\k} \rightarrow \D$ be a natural hyperbolic function. The function $F$ is of hyperbolic bounded variation if and only if the idempotent component functions $F_1:[a_1, b_1] \rightarrow \R$ and $F_2:[a_2, b_2] \rightarrow \R$ are functions of real bounded variation.
\end{corollary}
On account of this result the set of discontinuities for an natural hyperbolic function of bounded variation is well defined, which is due to the fact that a real function of bounded variation only has jump discontinuities and therefore the set of discontinuities is numerable, see \cite[Sec. 6.8]{a1974}.
\begin{lemma}
If $F:[\alpha, \beta]_{\k} \rightarrow \D$ is  a natural hyperbolic functions of bounded variation, then the set of discontinuities is the numerable union of perpendicular line segments to idempotent axes.
\end{lemma}
\begin{proof}
The components $F_1$ and $F_2$ from $F$ are real functions of bounded variation so, there exist two set $\{x_{1, n}\}_{n \in \N} \subset [a_1, b_1]$ and $\{x_{2, n}\}_{n \in \N} \subset [a_2, b_2]$ of with all discontinuities for $F_1$ and $F_2$ respectively.

For every point $y \in [a_2, b_2]$ and $n \in \N$, the point $x_{1, n}\e + y\ed$ is a point of discontinuity for $F$ (See Sec. \ref{sc1}). Thus, the set of discontinuities contains the union $\displaystyle \bigcup_{n \in \N} x_{1, n}\e + [a_2, b_2]\ed$.

Similarly, the set of discontinuities of $F$ contains the union $\displaystyle \bigcup_{n \in \N} [a_1, b_1]\e + x_{2, n}\ed$.

The union $\displaystyle \frakD(F) = \left( \bigcup_{n \in \N} x_{1, n}\e + [a_2, b_2]\ed \right) \cup \left( \bigcup_{n \in \N} [a_1, b_1]\e + x_{2, n}\ed \right)$ contains all discontinuities of $F$, because if there exist $\xi = x\e + y\ed$ a discontinuity of $F$, then $x$ is a discontinuity of $F_1$ or $y$ is a discontinuity of $F_2$, but this implies that $x \in \{x_{1, n}\}_{n \in \N}$ or $y \in \{x_{2, n}\}_{n \in \N}$.
\end{proof}

\begin{theorem}
The set of discontinuities from a natural hyperbolic function of bounded variation is of zero measure with the Lebesgue measure in the Euclidean plane.
\end{theorem}
\begin{proof}
It is a consequence that every line in the Euclidean plane has zero measure and numerable union of these set again has zero measure.
\end{proof}

This result can not be translated to hyperbolic Lebesgue measure defined in \cite{ks2017}. Since, the Lebesgue measure is defined as $\mu = \mu_{\R}\e + \mu_{\R}\ed$, where $\mu_{\R}$ is the Lebesgue measure in the real line, implies that the set $x_{1, n}\e + [a_2, b_2]\ed$ have not zero measure for all $n \in \N$, because $[a_2, b_2]$ is not a real set of zero measure.

\section{Hyperbolic Valued Riemann-Stieltjes Integral}\label{scHRSI}

Strong partitions can be applied to define a Riemann-Stieltjes type integral over hyperbolic valued functions.

The diameter of a real partition $P$ is defined as the maximum into the set of all lengths of successive intervals generated by $P$,
\[\diam(P) = \max \{ \lambda([p_j+1, p_j]) \ |\ j \in \{0, ..., n-1\} \}. \]

Although the diameter of a partition can be extended to strong partitions in the hyperbolic numbers plane in a direct way, it is not convenient for an extension of Riemann-Stieltjes integral because even if the maximum length is taken, its projections could not coincide with the diameter of the projections of the strong partition.

\begin{definition}\label{d7}
Let $\frakP$ be a strong partition of $[\alpha, \beta]_{\k}$. The diameter of $\frakP$ is defined to be the hyperbolic number
\[\diam_{\k}(\frakP) = \diam(\frakP_{\e})\e + \diam(\frakP_{\ed})\ed. \]
\end{definition}

Although, Riemann-Stieltjes integral can be defined on general hyperbolic valued functions over an interval, our focus will be on the case of natural hyperbolic functions.

\begin{definition}\label{d6}
Let $F: [\alpha, \beta]_{\k} \rightarrow \D$ and $G: \D \rightarrow \D$ be two hyperbolic functions. An hyperbolic number $\calI$ is called the Riemann-Stieltjes integral of $F$ respect to $G$, if for every $\epsilon \in \D^{+}$ there exists a $\delta \in \D^{+}$ such that 
\[\left| S_{\k}(\frakP, F, G) -  \calI \right|_{\k} = \left| \sum_{j = 0}^{n_{\frakP} - 1} F(\gamma_{j}) \left| G(\rho_{j + 1}) - G(\rho_{j}) \right|_{\k} - \calI \right|_{\k} \prec \epsilon,\]
for any strong partition $\frakP \in \calP([\alpha, \beta]_{\k})$ that fulfill the property $\diam_{\k}(\frakP) \prec \delta$ and  whatever selection $\gamma_j \in [\rho_{j+1}, \rho_j]_{\k}$, with $j \in \{0, ..., n_{\frakP} - 1\}.$
\end{definition}

The quantity $S_{\k}(\frakP, F, G)$ is called the Riemann-Stieltjes sum. In addition, when such $\calI \in \D$ exists, it will be denoted by $\displaystyle \calI = \int_{\alpha}^{\beta} F d_{\k}G$.

When $F$ and $G$ in Def. \ref{d6} are assumed to be natural hyperbolic functions, the Riemann-Stieltjes sum is analogue to Eq. \ref{eq3} and hence
\[\begin{split}
S_{\k}&(\frakP, F, G) = \\
  &\left(\sum_{j = 0}^{n_{\frakP} - 1} F_{1}(y_{j, 1})\left| G_{1}(p_{j + 1, 1}) - G_1(p_{j, 1})\right| \right)\e + \left(\sum_{j = 0}^{n_{\frakP} - 1} F_{2}(y_{j, 2})\left| G_{2}(p_{j + 1, 2}) - G_2(p_{j, 2})\right| \right)\ed. 
\end{split}\]

But $F_1$ and $F_2$ are real valued functions defined on the respective projections of $[\alpha, \beta]_{\k}$, likewise $G_1$ and $G_2$ are real valued functions defined on the whole real line. Thus, the Riemann-Stieltjes over the hyperbolic plane is the sum of classic Riemann-Stieltjes sum on the partition generated by projections of $\frakP$ (see \cite[Sec. 7.3]{a1974}).
\[S_{\k}(\frakP, F, G) = S(\frakP_{\e}, F_1, G_1)\e + S(\frakP_{\ed}, F_2, G_2). \]

By definition of usual metric on $\D$ (for more details Sec. \ref{sc1}) and when natural hyperbolic functions are assumed, we have
\begin{equation}\label{eq-rsk}
\int_{\alpha}^{\beta} F d_{\k}G = \left(\int_{a_1}^{b_1} F_1 dG_1\right)\e + \left( \int_{a_2}^{b_2} F_2 dG_2 \right)\ed,
\end{equation}
where the integrals in the left side in Eq. \ref{eq-rsk} are classical real valued Riemann-Stieltjes integrals.

\begin{theorem}\label{tmERS}
An hyperbolic natural function $F:[\alpha, \beta]_{\k} \rightarrow \D$ is hyperbolic Riemann-Stieltjes integrable with respect to a natural hyperbolic function $G: \D \rightarrow \D$ if and only if every component $F_1:[a_1, b_1] \rightarrow \R$ and $F_{2}:[a_2, b_2] \rightarrow \R$ are real Riemann-Stieltjes integrable functions respect to $G_{1}:\R \rightarrow \R$ and $G_{2}: \R \rightarrow \R$.
\end{theorem}
\begin{proof}
The proof is followed by previous comments.
\end{proof}
\bigskip
\noindent
From now on, we make the assumption that $F$ and $G$ are natural hyperbolic functions.

Introduce the identity function $Id_{\k}:\D \rightarrow \D$, which is a natural hyperbolic function
\[Id_{\k}(\xi) = \xi = x_1\e+x_2\ed = Id(x_1)\e + Id(x_2)\ed.\]
There is a very close connection between the Riemann–Stieltjes integral and the Riemann integral we are aiming to classify. Indeed, the hyperbolic Riemann-Stieltjes integral of a natural hyperbolic function $F$ with respect to $Id_{\k}$ can be viewed as the hyperbolic Riemann integral introduced in \cite[Ch. IV]{vd1935} or a particular case of Lebesgue integral following \cite[Sec. 3]{ks2017}.

Results in \cite{vd1935} requires non-self-intersecting continuous loop (Jordan curve). Taking the straight line that joins the two extreme points of an hyperbolic interval we get a loop of this kind. Therefore, for every strong partition $\frakP$ the union of lines that join every sub-interval $[\rho_{j+1}, \rho_{j}]_{\k}$, where $j \in \{0, ..., n-1\}$, is a Jordan loop and
\[\int_{\alpha}^{\beta} F d_{\k}Id_{\k}  = \int_{\alpha}^{\beta} F d_{\k}\xi = \left(\int_{a_1}^{b_1} F_1 dx_1 \right) \e + \left( \int_{a_2}^{b_2} F_2 dx_2 \right) \ed. \]

Another way to consider the hyperbolic Riemann integral is through the use of an hyperbolic valued measure $\mu_{\k} := \mu_{\R}\e + \mu_{\R}\ed$. Taking into account the definition of Lebesgue integral in \cite{ks2017} and since the real Lebesgue integral restricted to a closed interval reduces to the Riemann integral, the relation 
\[\int_{\alpha}^{\beta} F d_{\k}\xi = \int_{[\alpha, \beta]_{\k}} F d\mu_{\k} = \left(\int_{[a_1, b_1]}F_{1}d\mu_{\R} \right)\e + \left( \int_{[a_2, b_2]} F_2d\mu_{\R} \right)\ed\]
holds.

Let us mention an important property of the hyperbolic Riemann-Stieltjes integral when the integrator is an holomorphic function.
\begin{theorem}\label{thRSR}
Let $G: \D \rightarrow \D$ be an holomorphic and continuously differentiable function, $F:[\alpha, \beta] \rightarrow \D$ a natural hyperbolic Riemann-Stieltjes integrable function with respect to $G$. Then
\[\int_{\alpha}^{\beta} F d_{\k}G =  \int_{\alpha}^{\beta} FG' d_{\k} \xi. \]
\end{theorem}
\begin{proof}
By holomorphic assumption on $G$, it is a natural hyperbolic functions with derivative $G'(\xi) = G_1'(x)\e + G_2'(y)\ed$, for $\xi \in \D$. Also, since $G$ is continuously differentiable its idempotent component have continuous derivatives of any order. Therefore $G_1$ and $G_2$ are functions of bounded variation and \cite[Thm. 7.8]{a1974} makes easy to see that
\[ \int_{a_1}^{b_1} F_1 dG_1 = \int_{a_1}^{b_1} F_1G_1' dx_1\quad \text{and} \quad \int_{a_2}^{b_2} F_2 dG_2 = \int_{a_2}^{b_2} F_2G_2' dx_2.\]
Combining these equalities the result is obtained.
\end{proof}

\begin{remark}
The continuously differentiability of $G$ can not be omitted from the hypotheses, because unlike happen in the complex analysis context, hyperbolic holomorphic functions does not have derivatives of all orders, see \cite{vd1935, mr1998, ks2005}.
\end{remark}

Although Theorem \ref{thRSR} establish a direct relation between Riemann-Stieltjes and Riemann integral, the integrability of $F$ respect to $G$ is required. So, it should be convenient to see under what conditions the integrability holds. 
\begin{theorem}
Suppose that $F: [\alpha, \beta]_{\k} \rightarrow \D$ is a continuous natural hyperbolic function and $G: \D \rightarrow \D$ is a natural hyperbolic function of bounded variation. Then $F$ is Riemann-Stieltjes integrable with respect to $G$.
\end{theorem}
\begin{proof}
Since components $F_1$ and $F_2$ of a continuous natural hyperbolic function $F$ are real continuous functions (see Sec. \ref{sc1}), Corollary \ref{crBVC} shows that $G_1$ and $G_2$ are real functions of bounded variation. Therefore, the integrals 
\[ \int_{a_1}^{b_1}F_1 dG_1 \quad and \quad \int_{a_2}^{b_2}F_2 dG_2\]
exist, which is clear from \cite[Thm. 7.27]{a1974}. Finally, by Theorem \ref{tmERS}, the result is obtained.
\end{proof}

\section*{Declarations}
\subsection*{Funding} Instituto Polit\'ecnico Nacional (grant number SIP20211188) and Postgraduate Study Fellowship of the Consejo Nacional de Ciencia y Tecnolog\'ia (CONACYT) (grant number 774598).
\subsection*{Conflict of interest} No conflict of interest regarding the work is reported in this paper.
\subsection*{Author contributions} G. Y. T\'ellez-Sanchez: Visualization, Investigation, Writing- Original draft preparation, J. Bory-Reyes: Conceptualization, Supervision, Writing- Reviewing and Editing. Both authors approved the final form of the manuscript.
\subsection*{Availability of data and material} No data were used to support this study.
\subsection*{Code availability} Not applicable.

\end{document}